\documentclass[12pt,a4paper]{amsart}
\usepackage{mathrsfs}
\usepackage{amsfonts}

\usepackage[active]{srcltx}
\usepackage{enumerate}
\usepackage{color}
\usepackage[all]{xy}
\usepackage{color}
 \usepackage[colorlinks,final,backref=page,hyperindex]{hyperref} 
\usepackage{amsmath,amssymb,xspace,amsthm}
\newtheorem{theorem}{Theorem}[section]

\newtheorem{remark}[theorem]{Remark}
\newtheorem{lemma}[theorem]{Lemma}
\newtheorem{proposition}[theorem]{Proposition}
\newtheorem{corollary}[theorem]{Corollary}
\newtheorem{conjecture}{Conjecture}[section]

\numberwithin{equation}{section}
\newtheorem{definition}[theorem]{Definition}
\input xy
\xyoption{arrow} \xyoption{matrix}

\def\Ext{\mathrm{Ext}}

\def\hom{\mathrm{Hom}}
\def\mh{\mathfrak{h}}
\def\b{\mathfrak{b}}
\newcommand{\C}{\mathbb C}

\newcommand{\Z}{\mathbb{Z}}
\newcommand{\cd}{\mathfrak{D}}

\newcommand{\W}{\mathfrak{W}}

\newcommand{\mb}{\mathfrak{b}}

\title[BGG  Category]{\bf  Category  $\mathcal{O}$  for the Lie algebra of vector fields on the line}
\author{Genqiang Liu and Mingjie Li}

\date{\today}

\begin{document}
\begin{abstract} Let $\mathfrak{W}$ be the Lie algebra of vector fields on the line. Via computing extensions between all simple modules in the category  $\mathcal{O}$, we  give the block decomposition of $\mathcal{O}$, and show that the representation type of each block of $\mathcal{O}$ is wild using the Ext-quiver. Each block of $\mathcal{O}$ has infinite simple objects. This result is very different from that of  $\mathcal{O}$ for  complex semisimple Lie algebras. To find a connection between $\mathcal{O}$ and the module category over some associative algebra, we define a subalgebra $H_1$ of $U(\mb)$.
We give an exact functor   from $\mathcal{O}$ to the category  $\Omega$ of finite dimensional modules over $H_1$.  We also construct new simple $\mathfrak{W}$-modules from Weyl modules and modules over the Borel subalgebra $\mathfrak{b}$ of $\mathfrak{W}$.
\end{abstract}
\vspace{5mm}
\maketitle

\vspace{2mm}

\noindent{{\bf Math. Subj. Class.} 2020: 17B10, 17B30, 17B80}

\section{introduction}

The category $\mathcal{O}$   for complex semisimple Lie algebras was
introduced by  Joseph Bernstein, Israel Gelfand and Sergei Gelfand in the early 1970s, see \cite{BGG},  and it includes all highest weight modules.
This category is very important in the representation theory.  For more details on category $\mathcal{O}$,  one can see the
 monograph \cite{Hu}.

The category $\mathcal{O}$ can be defined for any Lie algebra with a triangular decomposition, see the book \cite{MP}.
For a Lie algebra $\mathfrak{g}$ with a triangular decomposition $\mathfrak{g}=\mathfrak{g}^-\oplus\mathfrak{h}\oplus \mathfrak{g}^+$, there  always exists an anti-involution $\sigma$ of $\mathfrak{g}$ such that $\sigma(\mathfrak{g}^+)=\mathfrak{g}^-$ and $\sigma|_{\mathfrak{h}}=\text{id}_{\mathfrak{h}}$. For example, the finite dimensional simple Lie algebras, the Virasoro algebra, the affine Kac-Moody algebras, the Heisenberg Lie algebras are all  Lie algebras with triangular decompositions, see \cite{MP}. For the Kac-Moody algebras and the Virasoro algebra, categories $\mathcal{O}$ were studied in \cite{F,BNW2} and references therein.  For these algebras, the Hom-spaces between Verma modules determine  the block decomposition of
the  category $\mathcal{O}$ to a great extent.

For the  Lie algebra $\mathfrak{W}=\mathfrak{W}^-\oplus \mh\oplus\mathfrak{W}^+$  of vector fields on the line, $\mathfrak{W}^-$ is one dimensional, $\mathfrak{W}^+$ is infinite dimensional. So $\mathfrak{W}$ is not a  Lie algebra with a triangular decomposition in the sense of \cite{MP}. Although we can also define the category $\mathcal{O}$ for
$\mathfrak{W}$ similar as that of complex semisimple Lie algebras, however several properties for $\mathcal{O}$  in \cite{MP} dose not hold
for $\mathfrak{W}$. For example,  the embeddings between Verma modules has little impact on the block decomposition of $\mathcal{O}$. Our initial motivation of the present paper was to explore the differences between the category $\mathcal{O}$ of $\W$ and the categories $\mathcal{O}$ of semi-simple Lie algebras.
In this paper, through  giving  extensions between all simple modules in $\mathcal{O}$, we obtain the block decomposition of the category  $\mathcal{O}$ for $\W$, and study the representation type of each block of $\mathcal{O}$. We also detect the relation between $\mathcal{O}$ and the module category for some associative algebra, and construct simple $\W$-modules  from modules over the Weyl algebra and modules over the Borel subalgebra $\mb$ of $\W$.

 The paper is organized as follows. In Section 2, we introduce the
 category $\mathcal{O}$  for the Lie algebra $\W$.
 In Section 3, we first study the Verma modules and recall extensions on the $\W$-modules $F_{\lambda}$ of Feigin and Fuchs defined in \cite{FF}. Then using these extensions and the duality between $F_{\lambda}$ and the Verma module $\Delta(\lambda)$, we can give  all nontrivial extensions  between Verma modules in $\mathcal{O}$. Consequently, we obtain
 $\text{Ext}^1_{\mathcal{O}}(M,N)$  for all simple  modules $M, N\in \mathcal{O}$, see Theorem \ref{the}. It should be mentioned that extensions between simple modules for the finite dimensional Witt algebra $W(1, 1)$ over an algebraically closed field of characteristic $p>3$  were determined in \cite{BNW1}.
  Furthermore we give the block decomposition $\mathcal{O} =\oplus_{\lambda\in \C/\Z}\mathcal{O}_{[\lambda]}$, and show that each block $\mathcal{O}_{[\lambda]}$ is wild by studying a sub-quiver of its
 Ext-quiver, see Theorem \ref{wild1}. Let $H_1=\{u\in U(\mb)\mid u(d_{-1}-1)\subset (d_{-1}-1)U(\W)\}$
which is a subalgebra of $U(\mb)$. Moreover $H_1$ is isomorphic to the endomorphism algebra of an induced right $U(\mathfrak{W})$-module $Q'_1$ which is the universal Whittaker module defined in \cite{ZL}.  In subsection \ref{functor}, we construct a functor $\Gamma$ from $\mathcal{O}$ to the category  $\Omega_1
$ of finite dimensional $H_1$-modules. We show that $\Gamma$ is an exact functor. At the end of Section 3, we also conjecture that some non-integral block  $\mathcal{O}_{[\lambda]}$ may be equivalent to some
subcategory of  $\Omega_1$.
 In Section 4, we construct new simple tensor $\W$-modules $T(P,V)$ from modules $P$ over the Weyl algebra and $\mb$-modules $V$. The isomorphism criterion for $T(P,V)$ is also given.

\section{Preliminaries}

In this paper, we denote by $\Z$, $\Z_{>0}$, $\Z_{\geq 0}$ and $\C$ the sets of integers, positive integers, nonnegative
integers and complex numbers, respectively. All vector spaces and Lie algebras are over $\C$. For a Lie algebra
$\mathfrak{g}$ we denote by $U(\mathfrak{g})$ its universal enveloping algebra. We write $\otimes$ for
$\otimes_{\mathbb{C}}$.

\subsection{Witt algebra}

Let $A=\C[x]$ be the polynomial algebra
and $\mathfrak{W}$  the derivation Lie algebra of $A$,
i.e., $\mathfrak{W}=\text{Der}_\C A$.
The Lie algebra $\mathfrak{W}$ is called the Lie algebra of vector fields on the line, or the Witt algebra of rank one. Denote
$\partial=\frac{\partial}{\partial x}$ and $d_i=x^{i+1}\partial$, for any $i\in \Z_{\geq -1}$.
Then $\{d_i\mid i\in \Z_{\geq -1}\}$ is a basis of $\mathfrak{W}$.
We can write the Lie bracket in $\mathfrak{W}$  as follows:
$$[d_i,d_j]=(j-i)d_{i+j}, \ \text{for all}\  i,j \in\Z_{\geq -1} . $$
Note that the subspace $\mh=\C d_0$ is a Cartan subalgebra of $\mathfrak{W}$, i.e.,
a maximal
abelian subalgebra that is diagonalizable on $\mathfrak{W}$ with respect to the adjoint action. Let $\mathfrak{W}^+=\text{span}\{d_i\mid i\in \Z_{>0}\}$ and $\mathfrak{W}^-=\C d_{-1}$. Then
$\mathfrak{W}=\mathfrak{W}^-\oplus \mh\oplus\mathfrak{W}^+$ is a  decomposition of $\mathfrak{W}$, and the Lie subalgebra $\b:=\mh\oplus \mathfrak{W}^+$ is  called a Borel subalgebra of $\mathfrak{W}$.

One can see  that $\Phi=\{\varepsilon_{-1}, \varepsilon_1, \varepsilon_2, \cdots\}$
  is the root system of $\mathfrak{W}$, where $\varepsilon_i\in\mh^*$ such that
  $\varepsilon_i(d_0)=i$, $i\in \Z_{\geq -1}$. The subalgebra $\C d_{-1}\oplus\C d_0\oplus \C d_1 \cong \mathfrak{sl}_2$, and $z=-d_1d_{-1}+d_0^2-d_0$ is its Casimir element, i.e., $z$ is a central element in $U(\mathfrak{sl}_2)$.

\begin{definition}A left $U(\mathfrak{W})$-module  $M$ is called a {\it weight module} if $d_0$ acts diagonally on  $M$, i.e.,
$$ M=\oplus_{\lambda\in \C} M_\lambda,$$
where $M_\lambda:=\{v\in M \mid d_0v=\lambda v \}.$  For a weight module $M$,   denote $$\mathrm{Supp}(M):=\{\lambda\in \C \mid M_\lambda\neq0\}.$$
\end{definition}

If $M$ is a simple weight  $\W$-module, then $\text{Supp}(M)\subset \lambda +\Z$ for some $\lambda\in \C$. For a $\lambda\in\mathrm{Supp}(M)$, a nonzero vector
$v\in M_\lambda$ is called a maximal  vector if $\mathfrak{W}^+v=0$.  A weight module  is called a highest weight module if it is generated by
 a maximal weight vector.

We use $U(\W)$-$\text{Mod}$ to denote the category of all left $U(\W)$-modules.

\subsection{Category $\mathcal{O}$ }
Next we introduce the   category $\mathcal{O}$  for $\W$.
 \begin{definition}\label{o-def}
The   category $\mathcal{O}$  for $\W$ is a full subcategory of $U(\W)$-$\text{Mod}$
whose objects are $\W$-modules $M$ satisfying the following axioms:
 \begin{enumerate}[$($a$)$]
 \item $M$ is a finitely generated  $U(\W)$-module;
\item  $M$ is a weight module;
\item $M$ is locally $\W^+$-finite: for each $v\in M$, the subspace $U(\W^+)v$ is finite dimensional.
\end{enumerate}
\end{definition}

Let $M$ be a module in $\mathcal{O}$. By (a) and (c) in Definition \ref{o-def}, we can assume that $M$ is generated by a finite dimensional $U(\W^+)$-module $N$. By induction on the dimension of $N$,  we can show that $M$ has the following property.

\begin{lemma}\label{filtration}Any module $M$ in $\mathcal{O}$ has  a finite filtration of submodules as follows:
$$0 = M_0 \subset  M_1\subset \cdots  \subset M_m = M,$$
where each factor $M_j/M_{j-1}$ for $1 \leq j\leq m$ is a highest weight module.
\end{lemma}

So highest weight modules are basic constituents of
$\mathcal{O}$.

\section{Block decomposition of $\mathcal{O}$ }

In this section, we study extensions between Verma modules and simple modules in $\mathcal{O}$. Using the Ext-quiver, we show that each block of $\mathcal{O}$ has wild representation type. We also construct an exact functor  from $\mathcal{O}$ to  the  category $\Omega_1$ of finite dimensional modules over $H_1$.

\subsection{The Verma modules}
For a $\lambda\in \mathbb{C}$,  denote by $\mathbb{C}_{\lambda}$ the one-dimensional $\mathfrak{b}$-module
with the generator $v_{\lambda}$ and the action given by
\begin{displaymath}
\W^+v_{\lambda}=0, \ \  d_0 v_{\lambda}=\lambda v_\lambda.
\end{displaymath}

The {\em Verma module} over $\W$ is defined as follows:
\begin{displaymath}
\Delta( \lambda):=\mathrm{Ind}_{\b}^{\W}\mathbb{C}_{\lambda}\cong
U(\W)\bigotimes_{U(\b)}\mathbb{C}_{\lambda}.
\end{displaymath}

The  module $\Delta( \lambda)$ has  the
 unique simple quotient module $L( \lambda)$. By Lemma \ref{filtration}, the modules $L(\lambda)$  for $\lambda\in \C$ provide a complete set of irreducible modules in category $\mathcal{O}$.

\begin{lemma}\label{Verma}
\begin{enumerate}[$($1$)$]
\item $z v=\lambda(\lambda+1)v$ for all $v\in \Delta(\lambda)$.
\item The module $\Delta(\lambda)$ is simple  if and only  if $\lambda\neq 0$. So $ \Delta(\lambda) =L(\lambda)$ for $\lambda \neq 0$.
\item The module $\Delta(0)$ is a uniserial module whose structure can be described by the following exact sequence: \begin{equation}\label{seq}
0 \rightarrow \Delta(-1) \xrightarrow{\alpha} \Delta(0) \xrightarrow{\beta}  L(0) \rightarrow 0.
\end{equation}
\end{enumerate}
\end{lemma}
\begin{proof}
(1) The proof follows from $[z, d_{-1}]=0$, $\Delta(\lambda)=\C[d_{-1}]v_{\lambda}$ and $z v_{\lambda}=\lambda(\lambda+1)v_{\lambda}$.

(2) For any  $i\in \Z_{\geq 0}$, denote $v_{\lambda-i} :=d_{-1}^i\cdot v_\lambda$.
We can deduce that $d_0\cdot v_{\lambda-i}=(\lambda-i)v_{\lambda-i}$. Since $\Delta(\lambda)$ is generated by $v_\lambda$, $\Delta(\lambda)$ is reducible if and only if there is an $i\in \Z_{>0}$ such that
$d_1v_{\lambda-i}=d_2v_{\lambda-i}=0$.

We can compute
\begin{align*}
d_1\cdot v_{\lambda-i}& =d_1\cdot d_{-1}^i\cdot v_\lambda\\&=([d_1,d_{-1}^i]+d_{-1}^id_1)\cdot v_\lambda  \\
            &= \sum_{t=0}^{i-1}d_{-1}^t[d_1,d_{-1}]d_{-1}^{i-t-1}\cdot v_\lambda \\
            &=-2\sum_{t=0}^{i-1}d_{-1}^td_0d_{-1}^{i-t-1}\cdot v_\lambda\\
           &=-2\sum_{t=0}^{i-1}(\lambda-i+t+1) v_{\lambda-i+1}\\
            &=i(i-1-2\lambda)v_{\lambda-i+1}.
\end{align*}
Similarly,
\begin{align*}
  d_2\cdot v_{\lambda-i}
            =i(i-1)(3\lambda-i+2)v_{\lambda-i+2}.
\end{align*}

Consider $d_1\cdot v_{\lambda-i}=0$ and $d_2\cdot v_{\lambda-i}=0, i\in \Z_{>0}$,  we have the equations
\begin{equation*}
\begin{cases}
i(i-1-2\lambda)=0 ,\\
i(i-1)(3\lambda-i+2)=0.
\end{cases}
\end{equation*}
The solution of this equation is $\lambda=0,i=1$. Moreover when $\lambda=0$, the submodule generated by $v_{-1}$ is a proper submodule which is isomorphic to $\Delta(-1)$.
Thus $\Delta(\lambda)$ is simple if and only if $\lambda\neq0.$

(3) By (2), the submodule $N$ generated by $v_{-1}$ of $\Delta(0)$  is simple and $\Delta(0)/N\cong L(0)$. So $N$ is  the unique nontrivial submodule of $\Delta(0)$. Hence $\Delta(0)$ is a uniserial module.
\end{proof}

\begin{remark}\label{basis} If we denote $e_{\lambda-i} :=\frac{(-1)^i}{i!}d_{-1}^i\cdot v_\lambda$, for any  $i\in \Z_{\geq 0}$, then $\{e_{\lambda-i} : i\in \Z_{\geq 0}\}$ is also a basis of $\Delta(\lambda)$ such that the action of $\W$ on
$\Delta(\lambda)$ is defined as follows:
\begin{equation}d_k e_{\lambda-i}=\big((k+1)\lambda+k-i\big)e_{\lambda-i+k}, \  \forall\ \ k\in\Z_{\geq -1}, \end{equation}
where $e_{\lambda-i+k}=0$  when $k-i>0$.
\end{remark}

\begin{corollary}\label{finite}
Any module $M$ in $\mathcal{O}$ has finite composition length.
\end{corollary}
\begin{proof} By lemma \ref{filtration}, $M$ has a filtration
$$0 = M_0 \subset  M_1\subset \cdots  \subset M_m = M,$$
 such that each factor $M_j/M_{j-1}$ for $1 \leq j\leq m$ is a highest weight module. By Lemma \ref{Verma}, every Verma module has
 finite composition length, so does any highest weight module. Consequently the  composition length of $M$ is
 finite.
\end{proof}

\begin{remark} In \cite{DSY}, the authors defined another category $\mathcal{O}'$ for $\W$. Any module $M$ in $\mathcal{O}'$ is locally finite
over $\C d_{-1}\oplus \C d_0$ rather than $\W^+$. Similar as  $\Delta(\lambda)$, define the $\W$-module: \begin{displaymath}
\Delta'( \lambda):=
U(\W)\bigotimes_{U(\C d_{-1}\oplus \C d_0)}\mathbb{C}'_{\lambda},
\end{displaymath} where $\mathbb{C}'_{\lambda}=\C v'_\lambda$ is the
 $U(\C d_{-1}\oplus \C d_0)$-module defined by $d_{-1}v'_\lambda=0, d_0v'_\lambda=\lambda v'_\lambda$. Since any simple weight
 module over $\W$ has one dimensional weight spaces, see \cite{M}, $\Delta'( \lambda)$ does not has  finite composition length. So $\mathcal{O}'$ does not satisfy Corollary \ref{finite}.
\end{remark}

\subsection{Extension between Verma modules}
Recall that for $\W$-modules $M, N\in \mathcal{O}$, the first cohomology space $\text{Ext}^1_{U(\W)}(M,N)$ classifies the short exact sequences:
$0 \rightarrow N \xrightarrow{\alpha} K \xrightarrow{\beta}  M\rightarrow 0$, also called the extension of $N$ by $M$. Generally $K$ may not lie in $\mathcal{O}$.
 We are only interested in  that $K \in \mathcal{O}$, i.e., $K$ needs to be a weight module. 
 So $\text{Ext}^1_{\mathcal{O}}(M,N)\subset \text{Ext}^1_{U(\W)}(M,N)$. Note that $\mathcal{O}$ is closed under weight module extensions. That is $M, N\in \mathcal{O}$ and
 $K$ is a weight module, then $K\in \mathcal{O}$. 
 In this subsection, we will give all extensions between Verma modules in $\mathcal{O}$.

\begin{lemma}\label{Ext}
Let $\lambda, \mu \in \C$.
\begin{enumerate}[$($1$)$]
\item If $\lambda-\mu\in \Z_{\geq 0}$ and $M$ is  a highest weight module with the highest weight $\mu$,  then $\emph{Ext}_{\mathcal{O}}^1(\Delta(\lambda),M)=0;$
\item $\emph{Ext}_{\mathcal{O}}^1(\Delta(\lambda), \Delta(\lambda))=0$.
\end{enumerate}
\end{lemma}
\begin{proof}

(1) Suppose that
\begin{equation}\label{ll}
   0 \rightarrow M \xrightarrow{\alpha} N \xrightarrow{\beta}  \Delta(\lambda) \rightarrow 0
\end{equation}
  is a short exact sequence in $\mathcal{O}$, where $N$ is a weight module. As $\lambda-\mu\in \Z_{\geq 0}$, $\mathrm{Supp}(N)=\mathrm{Supp}(M)\cup\mathrm{Supp}( \Delta(\lambda)),$ so $\lambda$ is a maximum weight in $N$. Recall that  $\Delta(\lambda)$ is generated by the highest weight vector $v_\lambda$.
Since $\beta$ is surjective, there exists weight vector $0\neq v\in N_\lambda$, such that $\beta(v)=v_\lambda$.
Moreover,  $v$ must be a maximal vector.  Otherwise, there exists $d_i\in \W^+, 0\neq d_i\cdot v\in N_{\lambda+i}$ such that $\lambda+i\in \mathrm{Supp}(N)$, which contradicts to the maximality of $\lambda$. The map such that $v_\lambda\mapsto v$ can be extended to a $U(\W)$-module homomorphism $\beta^\prime :\Delta(\lambda)\rightarrow N$, and $\beta\beta^\prime=1_{\Delta(\lambda)}$. So the exact sequence (\ref{ll}) is split and hence $\text{Ext}_{\mathcal{O}}^1(\Delta(\lambda),M)=0$.

(2) is an immediate corollary of (1).
\end{proof}

Let us recall the $\W$-modules $F_{\lambda}$ of Feigin and Fuchs defined in \cite{FF}, with $ \lambda\in\C$.   The module $F_{\lambda}$ has a basis $\{f_j\mid j\in \Z_{\geq 0}\}$ with the $\W$-action defined by  $$d_if_j=(j-(i+1)\lambda)f_{i+j},$$ where
$i\in \Z_{\geq -1}, j\in \Z_{\geq 0}$. These modules are shown to be restricted dualities of Verma modules. For a weight  $\W$-module $V=\oplus_{\lambda} V_{\lambda}$,
the restricted duality $\W$-module $V^*=\oplus_{\lambda} \text{Hom}_{\C}( V_{\lambda},\C)$ is defined by the natural action:
$$(d_i\phi)(v)=\phi(-d_i v), $$ for all $ i\in \Z_{\geq -1}, \phi\in V^*, v\in V$. By the universal property of $\Delta(\lambda)$,
one can  check  that $F_{\lambda}^*\cong \Delta(\lambda)$. Feigin and Fuchs gave the classification of the extensions of $F_\mu$ by the modules $F_\lambda$.

\begin{proposition}\cite{FF}\label{Ext7} Suppose that $\lambda, \mu\in \C$. Then
\begin{equation*}
\emph{Ext}^1_{U(\W)}(F_\lambda, F_\mu)=
\begin{cases}
 \C,\  \text{if}\  \lambda-\mu=0,2,3,4;\\
 \C\oplus \C,\  \text{if}\ (\lambda,\mu)=(0,-1);\\
 \C, \ \text{if}\  \ (\lambda,\mu)=(0,-5)\ \text{or}\ (4,-1);\\
 \C, \ \text{if}\  \ (\lambda,\mu)=(\frac{5\pm \sqrt{19}}{2},\frac{-7\pm \sqrt{19}}{2});\\
 0, \ \text{otherwise}.
\end{cases}
\end{equation*}
\end{proposition}
Moreover  all nontrivial extensions  of  $F_\mu$ by $F_{\lambda}$ were listed in the table 1 on page 207 of
\cite{FF}. We recall these extensions  in a slightly different form as follows.

(1) The unique non-split extension $E(F_\lambda, F_\lambda)$ of  $F_\lambda$ by itself has a basis $\{f_j, f_j'\mid j\in \Z_{\geq 0}\}$
such that \begin{equation}\label{e1}\aligned d_if_j& =(j-(i+1)\lambda)f_{i+j},\\
d_if'_j& =(j-(i+1)\lambda)f'_{i+j}+(i+1)f_{i+j}.
\endaligned \end{equation}

(2) There are two non-split extensions: $E(F_0, F_{-1})$, $E'(F_0, F_{-1})$ of  $F_{-1}$ by $F_{0}$.   The module $E(F_0, F_{-1})$ has a
a basis $\{f_j, f_j'\mid j\in \Z_{\geq 0}\}$
such that $$\aligned d_if_j& =(j+i+1)f_{i+j},\\
d_if'_j& =jf'_{i+j}+(i+1)jf_{i+j-1}.
\endaligned $$
The module $E'(F_0, F_{-1})$ has a
a basis $\{f_j, f_j'\mid j\in \Z_{\geq 0}\}$
such that $$\aligned d_if_j& =(j+i+1)f_{i+j},\\
d_if'_j& =jf'_{i+j}+(i+1)if_{i+j-1}.
\endaligned $$

(3) The unique non-split extension $E(F_\lambda, F_{\lambda-2})$ of  $F_{\lambda-2}$ by $F_{\lambda}$  has a basis $\{f_j, f_j'\mid j\in \Z_{\geq 0}\}$
such that $$\aligned d_if_j& =(j-(i+1)(\lambda-2))f_{i+j},\\
d_if'_j& =(j-(i+1)\lambda)f'_{i+j}+\big((i+1)i(i-1)+2(i+1)ij\big)f_{i+j-2}.
\endaligned $$

(4) The unique non-split extension $E(F_\lambda, F_{\lambda-3})$ of  $F_{\lambda-3}$ by $F_{\lambda}$  has a basis $\{f_j, f_j'\mid j\in \Z_{\geq 0}\}$
such that $$\aligned d_if_j& =(j-(i+1)(\lambda-3))f_{i+j},\\
d_if'_j& =(j-(i+1)\lambda)f'_{i+j}+\big((i+1)i(i-1)j+(i+1)ij(j-1)\big)f_{i+j-3}.
\endaligned $$

(5) The unique non-split extension $E(F_\lambda, F_{\lambda-4})$ of  $F_{\lambda-4}$ by $F_{\lambda}$  has a basis $\{f_j, f_j'\mid j\in \Z_{\geq 0}\}$
such that $$\aligned d_if_j& =(j-(i+1)(\lambda-4))f_{i+j},\\
d_if'_j& =(j-(i+1)\lambda)f'_{i+j}+\big(\frac{(i+1)!}{(i-4)!}\lambda+\frac{(i+1)!j}{(i-3)!}\\
& \quad  -6(i+1)i(i-1)j(j-1)-4(i+1)ij(j-1)(j-2) \big)f_{i+j-4}.
\endaligned $$

(6) The unique non-split extension $E(F_0, F_{-5})$ of  $F_{-5}$ by $F_{0}$  has a basis $\{f_j, f_j'\mid j\in \Z_{\geq 0}\}$
such that $$\aligned d_if_j& =(j+5(i+1))f_{i+j},\\
d_if'_j& =jf'_{i+j}+\Big(2\frac{(i+1)!j}{(i-4)!}-5\frac{(i+1)!j(j-1)}{(i-3)!}\\
& \quad  +10(i+1)i(i-1)j(j-1)(j-2)+5(i+1)ij(j-1)(j-2)(j-3) \Big)f_{i+j-5}.
\endaligned $$

(7) The unique non-split extension $E(F_4, F_{-1})$ of  $F_{-1}$ by $F_{4}$  has a basis $\{f_j, f_j'\mid j\in \Z_{\geq 0}\}$
such that $$\aligned d_if_j& =(j+i+1)f_{i+j},\\
d_if'_j& =(j-4(i+1))f'_{i+j}+\Big(12\frac{(i+1)!}{(i-5)!}+22\frac{(i+1)!j}{(i-4)!}
+5\frac{(i+1)!j(j-1)}{(i-3)!}\\
& \quad  -10(i+1)i(i-1)j(j-1)(j-2)-5(i+1)ij(j-1)(j-2)(j-3) \Big)f_{i+j-5}.
\endaligned $$

(8) When $ (\lambda,\mu)=(\frac{5\pm \sqrt{19}}{2},\frac{-7\pm \sqrt{19}}{2})$, the unique non-split extension $E(F_\lambda, F_\mu)$ of  $F_{\mu}$ by $F_{\lambda}$  has a basis $\{f_j, f_j'\mid j\in \Z_{\geq 0}\}$
such that $$\aligned d_if_j& =(j-(i+1)\mu)f_{i+j},\\
d_if'_j& =(j-(i+1)\lambda)f'_{i+j}+\Big(\frac{(i+1)!(22\pm 5\sqrt{19})}{(i-6)!4}-\frac{(i+1)!j(31\pm 7\sqrt{19})}{(i-5)!2}\\
&\quad- \frac{(i+1)!j(j-1)(25\pm 7\sqrt{19})}{(i-4)!2}-\frac{(i+1)!j(j-1)(j-2)5}{(i-3)!}\\
& \quad  +5 (i+1)i(i-1)j(j-1)(j-2)(j-3)+  \frac{j!2(i+1)i}{(j-5)!}\Big)f_{i+j-6}.
\endaligned $$

In the above formulas, $f_j=f'_j=0$ if $j<0$. By the isomorphism $F_{\lambda}^*\cong \Delta(\lambda)$.
we obtain all nontrivial extensions  between Verma modules.

\begin{proposition}\label{Ext7} Suppose that $\lambda, \mu\in \C$. Then
\begin{equation*}
\emph{Ext}^1_{\mathcal{O}}(\Delta(\mu),\Delta(\lambda))=
\begin{cases}
 \C,\  \text{if}\  \lambda-\mu=2,3,4;\\
 \C, \ \text{if}\  \ (\lambda,\mu)=(0, -1), (0,-5)\ \text{or}\ (4,-1);\\
 \C, \ \text{if}\  \ (\lambda,\mu)=(\frac{5\pm \sqrt{19}}{2},\frac{-7\pm \sqrt{19}}{2});\\
 0, \ \text{otherwise}.
\end{cases}
\end{equation*}
\end{proposition}
\begin{proof}By $\dim \text{Ext}_{\mathcal{O}}^1(\Delta(\mu),\Delta(\lambda))\leq \dim \text{Ext}_{U(\W)}^1(\Delta(\mu),\Delta(\lambda))$ and $F_{\lambda}^*\cong \Delta(\lambda)$, we have
\begin{equation*}
\dim \text{Ext}^1_{\mathcal{O}}(\Delta(\mu),\Delta(\lambda))\leq
\begin{cases}
 1,\  \text{if}\  \lambda-\mu=0, 2,3,4;\\
 2,\  \text{if}\ (\lambda,\mu)=(0,-1);\\
 1, \ \text{if}\  \ (\lambda,\mu)=(0,-5)\ \text{or}\ (4,-1);\\
 1,\  \text{if}\  \ (\lambda,\mu)=(\frac{5\pm \sqrt{19}}{2},\frac{-7\pm \sqrt{19}}{2}).
\end{cases}
\end{equation*}

If $E(F_\lambda, F_\lambda)$ is a weight module, then there are nonzero $a_j\in \C$
such that $d_0(f'_j+a_jf_j)=(j-\lambda)(f'_j+a_jf_j) $ for almost all $j$. However on the other side,
$d_0(f'_j+a_jf_j)= (j-\lambda)f'_j+f_j+a_j(j-\lambda)f_j$ by (\ref{e1}), which is a contradiction.
So $E(F_\lambda, F_\lambda)$ is not a weight module, and hence $\text{Ext}_{\mathcal{O}}^1(\Delta(\lambda), \Delta(\lambda))=0$.
In fact, by (2) in Lemma \ref{Ext}, we can also see that $\text{Ext}_{\mathcal{O}}^1(\Delta(\lambda), \Delta(\lambda))=0$.

 Similarly $E(F_0, F_{-1})$ is not  a weight module. By the action of $d_0$ on $f'_j$, we can see that $E'(F_0, F_{-1})$,  $E(F_\lambda, F_{\lambda-2})$, $E(F_\lambda, F_{\lambda-3})$, $E(F_\lambda, F_{\lambda-4})$, $E(F_0, F_{-5})$, $E(F_4, F_{-1})$,
 and $E(F_{\frac{5\pm \sqrt{19}}{2}},F_{\frac{-7\pm \sqrt{19}}{2}})$ are weight modules. So these modules are also no-split extensions between Verma modules in $\mathcal{O}$. Then we can complete the proof.
\end{proof}

\subsection{Extensions between simple modules}

In this subsection, we compute
 $\text{Ext}^1_{\mathcal{O}}(M,N)$  for all simple  modules $M, N\in \mathcal{O}$.

\begin{lemma}If $\lambda-\mu\not\in \Z$, then $\emph{Ext}_{\mathcal{O}}^1(L(\lambda),L(\mu))=0$.
\end{lemma}

 \begin{proof}If $M\in \mathcal{O}$ such that $L(\mu)\subset M$ and $M/L(\mu)\cong L(\lambda)$, then $\text{Supp}(M)=\text{Supp}(L(\mu))\cup \text{Supp}(L(\lambda))$. Since
 $\lambda-\mu\not\in \Z$, $\text{Supp}(L(\mu))\cap \text{Supp}(L(\lambda))=\emptyset$. So $M\cong L(\mu)\oplus L(\lambda)$.
\end{proof}

\begin{lemma}\label{Ext6}
\begin{enumerate}[$($1$)$]
\item For all $\lambda\in\C$, $\dim \emph{Ext}_{\mathcal{O}}^1(L(\lambda),L(\lambda))=0$.
\item
We have $ \dim\emph{Ext}_{\mathcal{O}}^1(L(0),L(-1)) =1$. That is,  if
\begin{equation}0 \rightarrow \Delta(-1) \rightarrow M \rightarrow  L(0) \rightarrow 0,\end{equation} is a non-split exact sequence of $\W$-modules in $\mathcal{O}$, then $M\cong \Delta(0)$.
\item $ \dim\emph{Ext}_{\mathcal{O}}^1(L(0),L(\lambda)) =0$ for all $\lambda\in\C\backslash \{ -1\}$.
\end{enumerate}
\end{lemma}

\begin{proof}
(1) When $\lambda\neq0$, thanks to (1) in Lemma \ref{Ext} and $L(\lambda)=\Delta(\lambda)$,  we obtain  $\text{Ext}^1(L(\lambda),L(\lambda))=0$.    It is enough to prove that $\text{Ext}_{\mathcal{O}}^1(L(0),L(0))=0$.

Consider the short exact sequence $$0 \rightarrow \Delta(-1) \rightarrow \Delta(0) \rightarrow  L(0) \rightarrow 0,$$
where $\Delta(-1)$ is the unique nonzero submodule of $\Delta(0)$.
We can get a long exact sequence by using the functor $\text{Hom}_{\mathcal{O}}(-,L(0)):$ $$\cdots \rightarrow \text{Hom}_{\mathcal{O}}(\Delta(-1),L(0)) \rightarrow \text{Ext}_{\mathcal{O}}^1(L(0),L(0)) \rightarrow \text{Ext}_{\mathcal{O}}^1(\Delta(0),L(0))\rightarrow \cdots.$$
According to $(1)$ in Lemma \ref{Ext}, and the fact that $L(0)$ is not a composition factor of $\Delta(-1)$, we have $\text{Ext}^1_{\mathcal{O}}(\Delta(0),L(0))=0, \text{Hom}_{\mathcal{O}}(\Delta(-1),L(0))=0$, whence $\text{Ext}^1_{\mathcal{O}}(L(0),L(0))=0$.
Thus $\text{Ext}_{\mathcal{O}}^1(L(\lambda),L(\lambda))=0$ for any $\lambda\in\C$.

(2)
 Consider the short exact sequence
\begin{equation}\label{seq3.4}0 \rightarrow \Delta(-1) \rightarrow \Delta(0) \rightarrow  L(0) \rightarrow 0,\end{equation} where $\Delta(-1)$ is the maximal submodule of codimension $1$ of $\Delta(0)$.
According to Lemma \ref{Ext}, we have $\text{Ext}^1_{\mathcal{O}}(\Delta(0),L(-1))=0$. Applying
 $\text{Hom}_{\mathcal{O}}(-,L(-1)) $ to (\ref{seq3.4}), from $ \text{Hom}_{\mathcal{O}}(\Delta(0),L(-1))=0$,  we can get $$0 \rightarrow \text{Hom}_{\mathcal{O}}(\Delta(-1),L(-1)) \rightarrow \text{Ext}_{\mathcal{O}}^1(L(0),L(-1)) \rightarrow \text{Ext}_{\mathcal{O}}^1(\Delta(0),L(-1))\rightarrow\cdots.$$ Thus $\text{Ext}_{\mathcal{O}}^1(L(0),L(-1))\cong \text{Hom}_{\mathcal{O}}(\Delta(-1),L(-1))\cong \C$.

 (3) By (1) and (2), we can assume that $\lambda\neq 0, -1$.  Let  $M$ be a non-split extension of $L(\lambda)$ by $L(0)$ in $\mathcal{O}$. We can suppose that $L(\lambda)\subset M$ and $M/L(\lambda)= L(0)$.  Then $1\leq \dim M_0\leq 2$. There must exists $e_0'\in M_0\setminus L(\lambda)_0$ such that $d_1e'_0=0$. Thus $d_1d_{-1}e'_0=0$.
 As $z d_{-1}e'_0=\lambda(\lambda+1)d_{-1}e'_0=d_{-1}ze'_0=0$ and $\lambda\neq 0, -1$, we deduce that
 $d_{-1}e'_0=0$. Since $M$ is indecomposable, $d_2e'_0$ is a nonzero element in $L(\lambda)$. So
 $[d_{-1}, d_2]e'_0=0$ implies that $d_{-1}d_2e'_0=0$, contradicting with that $d_{-1}$ acts injectively on $L(\lambda)$.

\end{proof}

Next we use the relative Lie algebra cohomology to compute  $\dim\text{Ext}_{\mathcal{O}}^1(L(\lambda),L(0))$. For two weight $\W$-modules $M, N$,
$$\aligned \text{Ext}^1_{\W, \mh}(M,N)&\cong H^1(\W, \mh; \text{Hom}_{\C}(M,N))\\
 & \cong C^1(\W, \mh; \text{Hom}_{\C}(M,N))/B^1(\W, \mh; \text{Hom}_{\C}(M,N)),\endaligned$$ where the set of $1$-cocycles $C^1(\W, \mh; \text{Hom}_{\C}(M,N))$ is the the subspace of all $c\in \text{Hom}_{\mh}(\W, \mh; \text{Hom}_{\C}(M,N))$ such that
\begin{equation} c(\mh)=0,\ \  \ \ c([g_1,g_2])=[g_1, c(g_2)]-[g_2, c(g_1)],
\end{equation}
for all $g_1, g_2\in \W$, where  $[g, \psi]\in \text{Hom}_{\C}(M,N)$ such that $$[g, \psi](v)=g\psi(v)-\psi(gv),$$ for $g\in \W, \psi \in\text{Hom}_{\C}(M,N), v\in M $.
A $1$-cocycle $c$ is a coboundary if there is a $\psi\in \text{Hom}_{\mh}(M,N)$ such that $c(g)=[g, \psi]$ for any $g\in \W$.

\begin{lemma}\label{ext-1-0}
 $\dim\emph{Ext}_{\mathcal{O}}^1(L(\lambda),L(0)) =
\begin{cases}
 1,\  \text{if}\  \lambda=-1,-2;\\
 0,\  \text{if}\   \lambda\neq-1,-2.
\end{cases}$
\end{lemma}
\begin{proof}
 By (1) in Lemma \ref{Ext6}, we can assume $\lambda\neq 0$.
So $L(\lambda)=\Delta(\lambda)=U(\W)v_\lambda=\C[d_{-1}]v_\lambda$.
Suppose that  $L(0)=\C v_0$

According to (3.1.2) in \cite {K}, we have
\begin{align*}
  \Ext_{\mathcal{O}}^1(L(\lambda),L(0))
  &\cong \Ext_{(\W,\mh)}^1(\Delta(\lambda),L(0))\\
  &\cong\Ext_{(\W,\mh)}^1(U(\W)\otimes_{U(\b)}\C_\lambda,L(0))\\
  &\cong \Ext_{(\b,\mh)}^1(\C_\lambda,L(0))\\
  &\cong H^1(\b,\mh,\hom_{\C}(\C_\lambda,L(0))),
\end{align*}
where $\C_{\lambda}=\C v_\lambda$ is the one dimensional $\mb$-module.

 For $\omega\in C^1(\b,\mh;\hom_{\C}(\C_\lambda,L(0))),k,j\in\Z_{\geq 0}$, we have
\begin{equation}\label{ak}(j-k)\omega(d_{k+j})=[d_k,\omega(d_j)]-[d_j,\omega(d_k)].
\end{equation}

Taking $k=0$, by $\omega(d_0)=0$, we have
$ j\omega(d_{j})=[d_0,\omega(d_j)]$. After multiplying a suitable
scalar, we can assume that $\omega(d_j)(v_\lambda)=\delta_{\lambda+j, 0} v_0$. If $\lambda\in \Z_{\leq -3}$, then $\omega(d_1)=\omega(d_2)=\omega(d_{-1})=0$, hence
$\omega=0$ and $\dim\text{Ext}_{\mathcal{O}}^1(L(\lambda),L(0))=0$. If $\lambda=-2$, then $\omega(d_2)(v_{-2})=v_0$, $\omega(d_j)=0$ for any $j\neq 2$. So from $B^1(\b,\mh;\hom_{\C}(\C_\lambda,L(0)))=0$, we have  $\dim\text{Ext}_{\mathcal{O}}^1(L(-2),L(0))=1$. Similarly $\dim\text{Ext}_{\mathcal{O}}^1(L(-1),L(0))=1$.
\end{proof}

We can summarize the results on extensions of simple modules as follows:

\begin{theorem}\label{the}Suppose that $\lambda, \mu\in \C$. Then
\begin{equation*}
\emph{Ext}^1_{\mathcal{O}}(L(\mu),L(\lambda))=
\begin{cases}
 \C,\  \text{if}\  \lambda-\mu=2,3,4, \lambda\mu \neq 0;\\
 \C, \ \text{if}\  \ (\lambda,\mu)=(0, -1), (0,-2),(-1, 0)\ \text{or}\ (4,-1);\\
 \C, \ \text{if}\  \ (\lambda,\mu)=(\frac{5\pm \sqrt{19}}{2},\frac{-7\pm \sqrt{19}}{2});\\
 0, \ \text{otherwise}.
\end{cases}
\end{equation*}
\end{theorem}

It should be mentioned that extensions between simple modules for the finite dimensional Witt algebra $W(1, 1)$ over an algebraically closed field of characteristic $p>3$  were determined in \cite{BNW1}.

\subsection{Block decomposition of $\mathcal{O}$}

We first recall the notion of blocks of an abelian category $\mathcal{C}$. We assume that any object of $\mathcal{C}$ has  finite composition length. We introduce an equivalence relation on the set of isomorphism classes of
simple objects of $\mathcal{C}$ as follows: two simple objects $V, V '$ are equivalent if there exists
a sequence $V = V_1, V_2, \dots, V_r = V '$ of simple objects satisfying $\text{Ext}^1_{\mathcal{C}}(V_i, V_{i+1}) \neq 0$
or $\text{Ext}^1_{\mathcal{C}}(V_{i+1}, V_i) \neq 0$ for all $i$. Then for each equivalence class $\chi$, we denote by $\mathcal{C}_{\chi}$
the full subcategory of $\mathcal{C}$ consisting of objects whose all composition factors belong
to $\chi$. Each $\mathcal{C}_{\chi}$ is called a block of $\mathcal{C}$ and
\begin{equation}\label{block}\mathcal{C} =\oplus_{\chi}\mathcal{C}_{\chi}.\end{equation} Moreover, each $\mathcal{C}_{\chi}$ cannot be decomposed
into a direct sum of two nontrivial abelian full subcategories.
The decomposition in  (\ref{block}) is called the block decomposition of the category $\mathcal{C}$.

For any $\lambda\in \Z$, let $\mathcal{O}_{[\lambda]}$ be the full subcategory of $\mathcal{O}$ consisting of modules $M$ such that
$\text{Supp}M\subset \lambda+\Z$.

\begin{proposition}We have the block decomposition $\mathcal{O} =\oplus_{\lambda\in \C/\Z}\mathcal{O}_{[\lambda]}$, each $\mathcal{O}_{[\lambda]}$ is indecomposable. The set $\{L(\lambda+n)\mid n\in \Z\}$ is the set of  all simple modules in $\mathcal{O}_{[\lambda]}$.
\end{proposition}
\begin{proof} This result follows from Theorem \ref{the}.
\end{proof}

Let $\C\langle x_1, x_2\rangle$ be the free associative algebra over $\C$ in two variables $x_1, x_2$.  Recall that an abelian category $\mathcal{C}$ is wild if there exists an exact functor from the category of finite dimensional representations
 of the algebra $\C\langle x_1, x_2\rangle$ to $\mathcal{C}$ which preserves indecomposability and takes non-isomorphic modules to
 non-isomorphic ones, see Definition 2 in \cite{Mak}. The following Lemma is useful for the study of representations of infinite dimensional algebras. For its proof, one can see Proposition 2.1 in \cite{G}

  \begin{lemma}\label{wild} Let $\mathcal{C}$ be an abelian category. If the Ext-quiver of
  $\mathcal{C}$ contains a finite subquiver $Q$ whose underlying unoriented graph is neither a Dynkin nor an affine diagram such that two arrows in $Q$ can not  be concatenated, then $\mathcal{C}$ is wild.
\end{lemma}

 \begin{theorem}\label{wild1} For any $\lambda\in\C$, the block $\mathcal{O}_{[\lambda]}$ is wild.
\end{theorem}
\begin{proof} By Theorem \ref{the}, the Ext-quiver of every block  $\mathcal{O}_{[\lambda]}$  contains the following subquiver:
$$\xymatrix{
                &  L(\mu)               \\
 L(\mu-3) \ar[ur]\ar[dr]&  &    L(\mu-4)\ar[ul] \ar[dl] \ar[r]&L(\mu-2) \\
 &  L(\mu-1)  ,
     }$$
where $\mu\in \lambda+\Z$  such that $\mu( \mu-1)(\mu-2)(\mu-3)(\mu-4) \neq 0$.
This subquiver is  neither a Dynkin nor an affine diagram. So the  block  $\mathcal{O}_{[\lambda]}$ is wild by Lemma \ref{wild}.
\end{proof}

\begin{remark} We can compare the category $\mathcal{O}$ of $\W$ with the category $\mathcal{O}_{\mathfrak{sl}_2}$ of $\mathfrak{sl}_2$. Each non-regular block of $\mathcal{O}_{\mathfrak{sl}_2}$ is semi-simple. Every regular block of $\mathcal{O}_{\mathfrak{sl}_2}$ is
equivalent to the category of finite dimensional representations over $\mathbb{C}$ of
the following quiver
\begin{displaymath}\label{quiver}
\xymatrix{\bullet\ar@/^/[rr]^{a}&&\bullet\ar@/^/[ll]^{b}},\qquad ab=0.
\end{displaymath} So every block of $\mathcal{O}_{\mathfrak{sl}_2}$ is not wild.
It should be mentioned that representation types of all blocks of $\mathcal{O}$ for complex simple Lie algebras were independently obtained in \cite{BKM} and \cite{FNP}.
\end{remark}

\subsection{Relation between $\mathcal{O}$ and a subalgebra $H_1$ of $U(\mb)$}\label{functor}

In order to find the connection between $\mathcal{O}$ and finite dimensional  modules
 over some associative algebra, we define  $$H_1=\{u\in U(\mb)\mid u(d_{-1}-1)\subset (d_{-1}-1)U(\W)\},$$
which is a subalgebra of $U(\mb)$.

Let $N_{1}=\C v_1$ be the one dimensional right $\C[d_{-1}]$-module such that $v_1 \cdot d_{-1} = v_1 $. Let $Q'_{1}$ the induced right $U(\W)$-module from the right $\C[d_{-1}]$-module $N_{1}$, and
\begin{equation}\label{Ha}H'_{1}=\text{End}_{\W}(Q'_{1}).\end{equation}

\begin{lemma} The algebra $H_1 $ is isomorphic to the algebra $H'_1$.
\end{lemma}
\begin{proof}For any $u\in H_1$, we define a $\psi_u\in H'_{1}=\text{End}_{\W}(Q'_{1})$ such that $\psi_u(v_1)=v_1u$.
 Then one can check that the linear map $$\psi: H_1\rightarrow H'_1, u\mapsto \psi_u, $$ is an algebra isomorphism.

\end{proof}

%
%

We define a functor $\Gamma$  from $\mathcal{O}$ to the category of finite dimensional $H_1$-modules. For any module $M\in \mathcal{O}$, let
$$\Gamma(M)=M/(d_{-1}-1)M.$$
Clearly $\Gamma(M)$ is an $H_1 $-module. The following result is immediate.

\begin{lemma}\label{image}
$ \dim \Gamma(L(\lambda))=\begin{cases}
1, \ \text{if}\ \lambda\neq 0 ,\\
0, \ \text{if}\ \lambda= 0.
\end{cases}$
\end{lemma}

 Let $\Omega$ be the category of finite dimensional
$H_1$-modules.

\begin{theorem}The functor $\Gamma: \mathcal{O}\rightarrow \Omega$ is exact.
\end{theorem}
\begin{proof}Suppose that \begin{equation}
0 \rightarrow N \xrightarrow{\alpha} M \xrightarrow{\beta}  L \rightarrow 0,
\end{equation} is an exact sequence in $\mathcal{O}$. We will show that
\begin{equation}
0 \rightarrow \Gamma(N) \xrightarrow{\Gamma(\alpha)} \Gamma(M) \xrightarrow{\Gamma(\beta)}  \Gamma(L) \rightarrow 0,
\end{equation} is exact.

From the sujectivity of $\beta$ and $\Gamma(\beta)\Gamma({\alpha})=0$, we  see that
$\Gamma(\beta)$ is surjective and
 $\text{Im}\Gamma(\alpha)\subset \text{Ker}\Gamma(\beta)$.
We need to show that $\Gamma(\alpha)$ is injective and $\text{Ker}\Gamma(\beta)\subset\text{Im}\Gamma(\alpha) $.

For $n+(d_{-1}-1)N\in \text{Ker}\Gamma(\alpha)$, there exists some $m\in M$ such that  $\alpha(n)=(d_{-1}-1)m$. Then $(d_{-1}-1)\beta(m)=\beta\alpha(n)=0$. Since $L$ is a weight module, $d_{-1}-1$ acts injectively on $L$. Thus $\beta(m)=0$ and there is an $n_1\in N$ such that $m=\alpha(n_1)$. Consequently,
$\alpha(n)=(d_{-1}-1)\alpha(n_1)$. Then the injectivity of $\alpha$ implies that $n=(d_{-1}-1)n_1$, i.e., $n\in (d_{-1}-1)N$. So $\Gamma(\alpha)$ is injective.

For $m+(d_{-1}-1)M\in\text{Ker}\Gamma(\beta)$, we have that $\beta(m)=(d_{-1}-1)l$ for some $l\in L$. As $\beta$ is surjective, $\beta(m)=(d_{-1}-1)\beta(m')$ for some $m'\in M$, i.e.,
$m-(d_{-1}-1)m'\in \text{ker}\beta=\text{Im}\alpha$. So $m-(d_{-1}-1)m'=\alpha(n)$ for some $n\in N$, hence $\text{Ker}\Gamma(\beta)\subset\text{Im}\Gamma(\alpha)$.
Therefore $\Gamma: \mathcal{O}\rightarrow \Omega$ is exact.
\end{proof}



Denote the restriction of $\Gamma$ to $\mathcal{O}_{[\lambda]}$ by $\Gamma^{[\lambda]}$, and by $\Omega_{ [\lambda]}$ the subcategory of
$\Omega$ consisting  of the $H_1$-modules isomorphic to $\Gamma^{[\lambda]}(M)$ for $M\in \mathcal{O}_{[\lambda]}$. Finally we give a conjecture on  $\mathcal{O}_{[\lambda]}$.

\begin{conjecture}There is a $\lambda \not\in \Z$ such that $\mathcal{O}_{[\lambda]}$  is equivalent  to  $\Omega_{[\lambda]}$.
\end{conjecture}

\section{Tensor $\W$-modules from $\mb$-modules and Weyl modules}

Let $\cd$ be the  Weyl algebra of rank one, that is, $\cd$  is the associative algebra
over $\mathbb{C}$ generated by $x$,
$\partial$ subject to
the relation $[\partial,x]=1$. In this section, we construct simple $\W$-modules from $\cd$-modules and $\mb$-modules.

\subsection{Tensor module $T(P,V)$}

The following interesting algebra homomorphism was given   in \cite{XL1} and \cite{BIN} independently,  which plays an important role in the classification of simple weight modules for $\W_n$,  see \cite{XL1, GS}.

\begin{lemma}\label{hom} There is an algebra monomorphism $\phi$ from $U(\W)$ to
$\cd\otimes U(\mathfrak{b})$ such that
$$\aligned
d_{-1}&\mapsto \partial\otimes 1, \\
d_m &\mapsto x^{m+1}\partial\otimes 1+ \sum_{r=0}^{m}\binom{m+1}{r+1} x^{m-r} \otimes d_{r},\ m\in \Z_{\geq 0}.\endaligned$$
\end{lemma}

By Lemma \ref{hom}, for any $\cd$-module $P$, any $\mb$-module $V$,
the tensor product $P\otimes V$ can be made to be a $\W$-module denoted by $T(P,V)$.

 \begin{remark} Let $S=\C[x^{\pm 1}]/\C[x]$ which is a simple
 $\cd$-module. Then we have a functor $T(S,-)$ from the category $\mb\text{-mod}$ of finite
 dimensional $\mb$-modules to the category $\mathcal{O}$. We originally intended to use the functor $T(S,-)$  to study $\mathcal{O}$. However,
 in view of extensions between simple modules in $\mathcal{O}$, $T(S,V)$
 may be decomposable for some indecomposable $\mb$-module $V$. Nevertheless,
 we can still use the bifunctor $T(-,-)$ to construct simple $\W$-modules.
 \end{remark}

 In the following theorem, we will give the simplicity of $T(P,V)$ under some natural conditions.

\begin{theorem}\label{simple} Suppose that $P$ is a simple $\cd$-module, and $V$ is a simple $\mb$-module such that there is an $l\in \Z_{>0}$ satisfying
\begin{enumerate}[$($1$)$]
\item $d_l$ acts injectively on $V$;
\item $d_i
V=0$ for any $i>l$,
\end{enumerate}
then $T(P,V)$ is a simple $\W$-module.
\end{theorem}
\begin{proof} Let $N$ be a nonzero submodule of $T(P,V)$. Suppose that $u=\sum_{n=0}^q  p_n\otimes v_n$ is a nonzero element in $N$, where $p_0,\dots, p_q$ are linearly independent.

\

\noindent{\bf Claim 1.} \label{claim1} For any $X\in\cd$,  we have that $\sum_{n=0}^q  X p_n\otimes d_l^2v_n\in N$.

For any $k$ with $k\geq 2l$ and any $m$ with $-1\leq m \leq 2l+1$, we can compute that
$$\aligned & d_{k-m}d_m(p\otimes v)=d_{k-m}\Big(d_mp\otimes v+ \sum_{r=0}^{m}\binom{m+1}{r+1} x^{m-r}p \otimes d_{r}v\Big)\\
=\ & ( d_{k-m}d_mp)\otimes v+ \sum_{r=0}^{m}\binom{m+1}{r+1} (d_{k-m}x^{m-r}p) \otimes d_{r}v\\
&+\sum_{s=0}^{k-m}\binom{k-m+1}{s+1}x^{k-m-s}d_mp\otimes d_sv\\
&+\sum_{s=0}^{k-m}\sum_{r=0}^{m}\binom{k-m+1}{s+1}\binom{m+1}{r+1} x^{k-r-s}p \otimes d_sd_{r}v\\
=\ & m^{2l+2}x^{k-2l}p\otimes\frac{d_l^2v}{((l+1)!)^2} +g(m),
\endaligned$$ where $g(m)$ is the term with degree of $m$ smaller than $2l+2$.  Consider the coefficient of  $m^{2l+2}$ in $d_{k-m}d_m u$. By letting $m=-1,0, 1, \dots, 2l+1$,  using the Vandermonde matrix, we deduce that $\sum_{n=0}^q  x^i p_n\otimes d_l^2v_n\in N$, for all $ i\in \Z_{\geq 0}$.
From the action of $d_{-1}$ on $N$, we see that  $\sum_{n=0}^q  \partial^j x^i p_n\otimes d_l^2v_n\in N$ for all $ i, j \in \Z_{\geq 0}$.
So $\sum_{n=0}^q  X p_n\otimes d_l^2v_n\in N$ for any $X\in\cd$. Then Claim 1 follows.

\

\noindent{\bf Claim 2.} \label{claim2}$P\otimes d_l^2 v_n\subset N$, for any $0\leq n\leq q$.

Since $P$ is a simple $\cd$-module, by the Jacobson density theorem, for any $p\in P$, there is a $X_n$ such that
$$X_n p_i=\delta_{n,i}p,\ \ i=0,\dots, q.$$
By Claim 1, we obtain that $P\otimes d_l^2v_n\subset N$, for any $0\leq n\leq q$. Claim 2 follows.

\

\noindent{\bf Claim 3.} $N=T(P,V)$. Hence $T(P,V)$ is simple.

Let $V_1=\{v\in V\mid P\otimes v\subset N\}$. By Claim 2 and that $d_{l}$ acts injectively on $V$, $V_1\neq 0$. For any $v\in V_1, p\in P$, taking $m=0,1\dots, l$ in
$$d_m(p\otimes v)=d_mp\otimes v+ \sum_{r=0}^{m}\binom{m+1}{r+1} x^{m-r}p \otimes d_{r}v,$$
we can see that $p\otimes d_0v, p\otimes d_1v, \dots, p\otimes d_l v\in N$. So $V_1$ is a nonzero $\mb$-submodule of $V$. The simplicity of $V$ forces that $V_1=V$. Then Claim 3 follows.

\end{proof}

\subsection{Isomorphism criterion for $T(P,V)$}

Next we  give  the following isomorphism criterion for $T(P,V)$.

\begin{theorem} Suppose that $P, P'$ are simple $\cd$-modules, $V, V'$ are simple $\mb$-modules  such that there are $l, s\in \Z_{> 0}$ satisfying
$d_l$ (resp. $d_s$) acts injectively on $V$ (resp. $V'$) and $d_i V=0$ (resp. $d_i V'=0$) for any $i>l$ (resp. $i>s$). Then $T(P, V)\cong T(P' , V')$ if and only if $P\cong P', l=s$ and $V\cong V'$.
\end{theorem}
\begin{proof}The sufficiency is obvious. Now suppose that
$$\psi: T(P, V) \rightarrow T(P',V')$$ is an isomorphism of $\W$-modules.
Let $p\otimes v$ be a nonzero element in $T(P, V)$.  Write
$$\psi(p\otimes v)=\sum_{n=0}^q p'_i\otimes v'_i\in T(P',V'),$$
where $p_0',\dots,p'_q$ are linearly independent.
Similar to  the proof of Claim 1 of Theorem \ref{simple}, comparing the
the highest degree of $m$ on both sides of  $$\psi(d_{k-m}d_m(p\otimes v))=d_{k-m}d_m\psi(p\otimes v),$$
we have that $l=s$ and
$$ \psi(Xp\otimes v)= \sum_{n=0}^q  X p_n\otimes d_l^2v_n, \ \ \forall\ X\in \cd.$$
By the Jacobson density theorem, there exists $Y\in \cd$ such that  $Yp_i=\delta_{i 0}p_0$. Then
$$ \psi(Yp\otimes v)=p_0\otimes d_l^2v_n.$$
Replacing $Yp$ by $p$, $p_0$ by $p'$ and $d_l^4v_n$ by $v'$, we have
$$ \psi(Xp\otimes v)=Xp'\otimes v', \ \ \forall\ X\in \cd.$$

Consequently $\psi_1: P\rightarrow P'$ satisfying $\psi_1(Xp)=Xp'$ is a well-defined map. Since $P, P'$ are simple $\cd$-modules,
$$\text{Ann}_{\cd}(p)=\text{Ann}_{\cd}(p'), \ \text{and}\  P\cong\cd/\text{Ann}_{\cd}(p)\cong P'. $$

Thus $\psi_1$ is a $\cd$-module isomorphism and
$$\psi(p\otimes v)=\psi_{1}(p)\otimes v', \ \ \forall\ p\in P.$$

Then from $\psi(d_m(p\otimes v))=d_m\psi(p\otimes v)$, we obtain that
$$\psi(p\otimes d_r v)=\psi_{1}(p)\otimes d_r v', \ \ \forall\ p\in P, \  r\in \Z_{\geq 0}.$$

So $$\psi(p\otimes y v)=\psi_{1}(p)\otimes y v', \ \ \forall\ p\in P, \forall\ y\in U(\mb).$$

Therefore we have $\text{Ann}_{U(\mb)}(v)=\text{Ann}_{U(\mb)}(v')$. The simplicity of $V$ and $V'$ implies that
$V\cong U(\mb)/\text{Ann}_{U(\mb)}(v)\cong V'$.

\end{proof}

\begin{remark} For each $r>0$, denote the quotient algebra $\mb/\langle d_{r+i}: i>0\rangle$ by $\mathfrak{a}_r$.
From Theorem \ref{simple}, we know that, to obtain new simple  $\W$-modules $T(P, V)$, it is enough to construct
infinite dimensional simple modules $V$ over $\mathfrak{a}_r$
for $r>0$ such that the action of $d_r$ on $V$ is injective.
Simple modules over $\mathfrak{a}_1$ and $\mathfrak{a}_2$ were  classified in \cite{MZ}.
\end{remark}

\vspace{4mm}

 \noindent G.L.: School of Mathematics and Statistics,
and  Institute of Contemporary Mathematics,
Henan University, Kaifeng 475004, China. Email: liugenqiang@henu.edu.cn

\vspace{0.2cm}

 \noindent M.L.: School of Mathematics and Statistics,
Henan University, Kaifeng 475004, China. Email: 13849167669@163.com

\end{document}